\documentclass[10pt,twoside]{siamart1116}
\usepackage[english]{babel}
\usepackage{graphicx,epstopdf,epsfig}
\usepackage{amsfonts,epsfig,fancyhdr,graphics, hyperref,amsmath,amssymb}
\usepackage{amsmath,amssymb}
\usepackage{graphicx}
\usepackage{url}
\usepackage[mathlines]{lineno}
\usepackage{dsfont} % needed for double-struck 1
\usepackage{color}
\usepackage{mathtools}
\usepackage{array}

\renewcommand{\theequation}{\thesection.\arabic{equation}}
\numberwithin{figure}{section}   % added LH 5/16/19
\numberwithin{table}{section}   % added LH 5/16/19

\newcommand*\patchAmsMathEnvironmentForLineno[1]{%
  \expandafter\let\csname old#1\expandafter\endcsname\csname #1\endcsname
  \expandafter\let\csname oldend#1\expandafter\endcsname\csname end#1\endcsname
  \renewenvironment{#1}%
     {\linenomath\csname old#1\endcsname}%
     {\csname oldend#1\endcsname\endlinenomath}}% 
\newcommand*\patchBothAmsMathEnvironmentsForLineno[1]{%
  \patchAmsMathEnvironmentForLineno{#1}%
  \patchAmsMathEnvironmentForLineno{#1*}}%
\AtBeginDocument{%
\patchBothAmsMathEnvironmentsForLineno{equation}%
\patchBothAmsMathEnvironmentsForLineno{align}%
\patchBothAmsMathEnvironmentsForLineno{flalign}%
\patchBothAmsMathEnvironmentsForLineno{alignat}%
\patchBothAmsMathEnvironmentsForLineno{gather}%
\patchBothAmsMathEnvironmentsForLineno{multline}%
}

\definecolor{red}{rgb}{1,0,0}

\definecolor{blue}{rgb}{0,0,1}

\definecolor{green}{rgb}{0,.6,0}

\newtheorem{thm}{Theorem}[section]
\newtheorem{cor}[thm]{Corollary}

\newtheorem{prop}[thm]{Proposition}
\newtheorem{conj}[thm]{Conjecture}
\newtheorem{obs}[thm]{Observation}

\theoremstyle{definition}

\theoremstyle{definition}

\theoremstyle{definition}

\newcommand{\F}{\mathcal{F}}
\newcommand{\osl}{\mathcal{S}}

\newcommand{\be}{{\bf e}}

\newcommand{\bone}{{\bf \mathds{1}}}
\newcommand{\bzero}{{\bf 0}}

\newcommand{\Z}{\operatorname{Z}}
\newcommand{\spec}{\operatorname{spec}}

\newcommand{\pt}{\operatorname{pt}}

\newcommand{\pr}{\mathbf{Pr}}
\newcommand{\e}{\mathbf{E}}
\newcommand{\var}{\mathbf{Var}}
\newcommand{\ept}{\operatorname{ept}}
\newcommand{\rad}{\operatorname{rad}} 
\newcommand{\sun}{\operatorname{-Sun}}
\newcommand{\comb}{\operatorname{-Comb}}

\newcommand{\tab}{$\null$\hspace{7mm}}

\newcommand{\bit}{\begin{itemize}}
\newcommand{\eit}{\end{itemize}}
\newcommand{\ben}{\begin{enumerate}}
\newcommand{\een}{\end{enumerate}}
\newcommand{\beq}{\begin{equation}}
\newcommand{\eeq}{\end{equation}}
\newcommand{\bea}{\begin{eqnarray*}} % * means no number
\newcommand{\eea}{\end{eqnarray*}}
\newcommand{\bpf}{\begin{proof}}
\newcommand{\epf}{\end{proof}\ms}
\newcommand{\bmt}{\begin{bmatrix}}
\newcommand{\emt}{\end{bmatrix}}
\newcommand{\ms}{\medskip}

\newcommand{\lf}{\left\lfloor}
\newcommand{\rf}{\right\rfloor}
\newcommand{\lp}{\!\left(}
\newcommand{\rp}{\right)}
\newcommand{\lb}{\!\left[}
\newcommand{\rb}{\right]}

\newcommand{\noi}{\noindent}

\title{Using Markov chains to determine expected propagation time for probabilistic zero forcing}
\author{Yu Chan\thanks{Department of Mathematics, Iowa State University, Ames, IA 50011, USA, (ychan,  ecurl, geneson, hogben, kevinliu, iodegard, msross)@iastate.edu.}\and Emelie Curl\footnotemark[1]\and  Jesse Geneson\footnotemark[1] \and  Leslie Hogben\footnotemark[1]\ \thanks{American Institute of Mathematics, 600 E. Brokaw Road, San Jose, CA 95112, USA, hogben@aimath.org} \and Kevin Liu\footnotemark[1]\and Issac Odegard\footnotemark[1]\and Michael Ross\footnotemark[1]}

\begin{document}
\maketitle
%\linenumbers

\begin{abstract} % edited 5/16/19 LH
 Zero forcing is a coloring game played on a graph where each vertex is initially colored blue or white and the goal is to color all the vertices
blue by repeated use of a (deterministic) color change rule starting with as few blue vertices as possible.  Probabilistic zero forcing yields a discrete dynamical system governed by a Markov chain. Since in a connected graph any one vertex can
eventually color the entire graph blue using probabilistic zero forcing, the expected
time to do this studied. Given a Markov transition matrix for a probabilistic zero forcing process, we establish an exact formula for expected propagation time.  We apply Markov chains to determine bounds on expected propagation
time for various families of graphs.  \end{abstract}

\noi {\bf Keywords} probabilistic zero forcing, expected propagation time, Markov chain

\noi{\bf AMS subject classification} 15B51, 60J10, 05C15, 05C57, 05D40, 15B48, 60J20, 60J22 % edited 5/16/19 LH

%%%%%%%%%%%%%%%%%%%%%%%%%%%%%%%%%%%%%%%%%%%
\section{Introduction}\label{sintro}

A {\em graph}, which can be used to model relationships between objects,  is a pair $G = (V,E)$. The set  $E = E(G)$ of {\em edges} (relationships) consists of $2$-element subsets of the set  $V = V(G)$ of {\em vertices} (objects).  Two vertices $v, w$ are  {\em adjacent}  if $\{v,w\} \in E$.
Suppose a graph $G$ is colored so that every vertex is blue or white. Vertices in the graph can change color based on the {\em zero forcing color change rule}: If a blue vertex $v$ is adjacent to exactly one white vertex $w$, then the white vertex changes to blue.   In this case, we say that $v$ {\em forces} $w$ and denote this by $v \to w$. A set of vertices $S$  is called a {\em zero forcing set} if when the vertices in $S$ are colored blue and those in $V\setminus S$ are colored white, repeated application of the color change rule forces all of the vertices to be blue. The {\em zero forcing number} of a graph $G$, denoted $\Z(G)$, is the minimum cardinality of a zero forcing set \cite{AIM08}. Throughout this paper, a force performed using the zero forcing color change rule is called a {\em deterministic force}.

Zero forcing was introduced in the study of the control of quantum systems by mathematical physicists who called it the ``graph infection number" \cite{BG07, Burg09}. Zero forcing was also introduced independently in the study of the  minimum rank problem in combinatorial matrix theory to bound the maximum nullity \cite{AIM08}. Zero forcing and its positive semidefinite  variant have been used extensively in the study of  the minimum rank problem (see \cite{FH14} and the references therein). %\cite{smallparam} \cite{Ber08} \cite{Edh12}  \cite{Huang10} \cite{John99}. 
Parameters derived from zero forcing have also been studied. Examples include propagation time (e.g.  %\cite{Chil12} 
\cite{Hog12, PSDpropTime}) and throttling (e.g. \cite{BY13}). Zero forcing also has connections to graph searching \cite{Y13} and power domination \cite{PDZ}. 

Two vertices are called  {\em neighbors} if they are adjacent, and the set of neighbors of a vertex $v$ in $G$ is denoted by $N(v)$. The {\em closed neighborhood} of a vertex $v$ is $N[v] = N(v) \cup \left\{v\right\}$.
A variant of zero forcing called \emph{probabilistic zero forcing} was introduced by Kang and Yi \cite{KY13} and is defined as follows: In one \emph{round}, each blue vertex $u$ attempts to \emph{force} (change the color to blue) each of its white neighbors $w$ independently with probability 
\[\pr(u \to w) = \frac{|N[u] \cap B|}{\deg u},\] where $B$ denotes the set of blue vertices. Because a vertex $u$ attempts to force each of its white neighbors independently, this action is a binomial (or Bernoulli) experiment with probability of success given by the previous formula. This color change rule is known as the {\em probabilistic color change rule}, and {\em probabilistic zero forcing} refers to the process of coloring a graph blue by repeated application of the probabilistic color change rule.

The study of probabilistic zero forcing therefore produces a discrete dynamical system that plausibly
describes many real world applications. Some of these applications include modeling the spread of a rumor through a social network, the spread of an infectious disease in a population, or the dissemination of a computer virus in a network. In addition, this type of zero forcing offers a new approach to coloring a graph. It should be noted that while for traditional zero forcing, the parameter of primary interest is the minimum number of vertices required to force the entire graph blue, in probabilistic zero forcing one blue vertex per connected component is necessary and sufficient to eventually color an entire graph blue. Therefore finding a minimum probabilistic zero forcing set is not an interesting problem. However, there are parameters related to probabilistic zero forcing that are of interest. 

One such parameter is expected propagation time, which is the focus of this paper. Suppose that $G$ is a connected graph with the vertices in $B \neq \emptyset$ colored blue and all other vertices white. The {\em probabilistic propagation time} of $B$, denoted by $\pt_{pzf}(G,B)$, is defined as the random variable equal to the number of the round in which the last white vertex turns blue when applying the probabilistic color change rule \cite{GH18-PZF}. %In this paper, all graphs are assumed to be connected. % (if $G$ is not connected, $Z$ is assumed to contain at least one vertex from each connected component of $G$). 
For a connected graph $G$ and a set $B \subseteq V(G)$ of vertices, the {\em expected propagation time of $B$} is the expected value of the propagation time of $B$ \cite{GH18-PZF}, i.e.,
\[\ept(G,B) = \textbf{E}[\pt_{pzf}(G,B)].\]
The {\em expected propagation time} of a connected graph $G$ is the minimum of the expected propagation time of $B$ over all one-vertex sets $B$ of $G$ \cite{GH18-PZF}, i.e.,
\[\ept(G) = \min \{\ept(G, \{v\}): v \in V(G)\}.\]

The use of Markov chains for probabilistic zero forcing was introduced in \cite{KY13} and studied further in \cite{GH18-PZF}. If $M$ is the $s \times s$ Markov matrix where the first state is one blue vertex and the last state is all vertices blue, then \[\ept(G,B)=\sum_{r=1}^\infty r\lp M^r- M^{r-1}\rp_{1s}\] \cite{GH18-PZF}.
In Section \ref{sMarkov} we provide an exact method to calculate $\ept(G, B)$ and apply it to obtain a table of the expected propagation times of small graphs. We also prove that there exist arbitrarily large graphs for which adding an edge increases the expected propagation time, answering a question in \cite{GH18-PZF}. This section also includes a characterization of the Markov matrix for the complete graph $K_n$ on $n$ vertices and data on its expected propagation time for various $n$. We also provide constructions of Markov matrices for complete bipartite graphs $K_{m, n}$, $n$-sun graphs, and $n$-comb graphs, as well as data on their behavior. 

In Section \ref{sBounds} we prove that $\ept(K_n) = \Theta(\log \log n)$, improving the upper bound given in \cite{GH18-PZF}, and $\ept(K_{c,n}) = \Theta(\log n)$, where $c \geq 1$ is a fixed integer. 
We prove that $\ept(G) = O(n)$ for any connected graph $G$ on $n$ vertices. Furthermore, we prove a $\Theta(\log n)$ bound on the expected propagation time of graphs on $n$ vertices obtained by adding a universal vertex to a graph of bounded degree. %A lower bound on $\ept(\rpr(G,H,v),(u,v))$ where $\rpr(G,H,v)$ is the rooted product of graph $G$ with graph $H$ at root $v$ is specified. In addition, another lower bound on a specialized form of $\rpr(G,H,v)$ is stipulated. 

We define some additional terms from graph theory and notation that we will use throughout the paper.  The \emph{order} of a graph is the number of vertices. The \emph{path} $P_n$ of order $n$ is a graph whose vertices can be listed in the order $v_1, \dots, v_n$ such that the edges of the graph are $\left\{v_i, v_{i+1}\right\}$ for $i = 1, \dots, n-1$. The \emph{cycle} $C_n$ of order $n$ is a graph whose vertices can be listed in the order $v_1, \dots, v_n$ such that the edges of the graph are $\left\{v_i, v_{i+1}\right\}$ for $i = 1, \dots, n-1$ and $\left\{v_1, v_n\right\}$. The \emph{complete graph} $K_n$ is the graph of order $n$ with all possible edges. The \emph{complete bipartite graph} $K_{m, n}$ is the graph of order $m+n$ whose vertices can be divided into two parts $u_1, \dots, u_m$ and $v_1, \dots, v_n$ such that the edges of the graph are $\left\{u_i, v_j \right\}$ for $1 \leq i \leq m$ and $1 \leq j \leq n$.   As a shorthand, we denote the edge $\left\{u, v \right\}$ as $uv$ (since the graphs in the paper are not directed, the same edge could be written as $v u$).  If $v$ is a vertex in $G$, then $G - v$ denotes the graph obtained from $G$ by removing the vertex $v$ and all edges that contain $v$. If $B$ is a set of blue vertices in $G$ and $v$ is a white vertex, we use $B \to v$ to denote that some vertex in $B$ forces $v$.

%%%%%%%%%%%%%%%%%%%%%%%%%%%%%%%%%%%%%%%%%%%

%%%%%%% begin Leslie's edits 5/18/19 %%%%%%%%%%%%
% new reference \bibitem{eptsmall} added

\section{Markov chains for probabilistic zero forcing}\label{sMarkov} 

In this section we introduce a method to compute expected propagation time exactly from the Markov transition matrix.  We then apply Markov chain methods to compute expected propagation time of small graphs and families of graphs.  We also answer the question of whether  adding an edge can raise expected propagation time (cf.  \cite[Question 2.16]{GH18-PZF}).

Let $G$ be a graph and $ B\subset  V(G)$ be nonempty.  A {\em simple  state for $B$} is a coloring of the vertices   that can be reached by starting with exactly the vertices in $B$ blue, and then applying the probabilistic color change rule iteratively.   
We normally combine simple states that behave analogously into one {\em state for $B$}. For example, in $K_n$ starting with one blue vertex, we use $n$ states, with state $k$ being the condition of having $k$ blue vertices.   
In most graphs, it matters which vertices are blue, and this is reflected by distinguishing states with the same number of blue vertices but different behavior.  

An  {\em ordered state list for $B$}, denoted by  $\osl=(S_1,\dots, S_s)$, is an ordered list of all  states  for $B$ in which $S_1$ is the initial state (where exactly the vertices in $B$ are blue), $S_s$ is the final state (where all vertices are blue), and the states $S_k,k=2,\dots,s-1$ are in some chosen order. A graph $G$ and an ordered  state list $\osl$ determine the Markov transition matrix for the process, which is denoted by $M(G,\osl)$.   Reordering the states $S_2,\dots,S_{s-1}$ results in a Markov transition matrix that is obtained by a permutation similarity of $M(G,\osl)$.
We use $|S_k|$ to denote the number of   blue vertices in state $S_k$, and  say $\osl$  is {\em properly ordered} if     $|S_i|<|S_j|$ implies  $i<j$.   

\begin{prop}\label{p:Markov-triangle}. % statement reordered
Let $G$ be a graph and let $ B\subset V(G)$ be nonempty.   Let $\osl$ be an ordered state list for $B$ and let $M(G,\osl)=[m_{ij}]$.  Then $\spec(M(G,\osl))=\{m_{kk}:k=1,\dots,s\}$, every eigenvalue is a real number in the interval $[0,1]$, and $1$ is a simple eigenvalue of $M(G,\osl)$. If $\osl$ %=(S_1,\dots, S_s)$ 
is a properly ordered state list for $B$,  then $M(G,\osl)$ is upper triangular.
\end{prop}
\bpf   Assume first that $\osl$ is properly ordered. If $i\ne j$ and it is possible to go from $S_i$ to $S_j$ in one round, then $|S_i|<|S_j|$ so $i<j$. Thus $M(G,\osl)$ is an upper triangular matrix and the eigenvalues are the diagonal entries.  The  probability $m_{kk}$ of remaining in state $S_k$ is less than one for $k<s$, is equal to one for $k=s$, and all $m_{kk}$ are nonnegative.  Thus, one is a simple eigenvalue and  is the spectral radius  of $M(G,\osl)$.  

Note that a permutation similarity does not change the eigenvalues of $M(G,\osl)$ or the (unordered) multiset of diagonal entries  (although the order of the diagonal entries may change).  Thus the  statements about the spectrum are true without the assumption that $\osl$ is properly ordered. %causes no loss of generality.
\epf  
%$\ept(G,Z)=\sum_{r=1}^\infty r\left(\lp M^r\rp_{1m}-\lp M^{r-1}\rp_{1m}\right)$

%As is customary, we use $\bone$ is a vector
\begin{thm}\label{t:ept-M} Suppose that $G$ is a graph,  $ B\subset V(G)$ is nonempty,  $\osl$ is an ordered state list for $B$ with $s$ states, and  $M=M(G,\osl)$. Then \[\ept(G,B)=((M-\bone{\be_s}^T-I)^{-1})_{1s}+1,\]
where $\bone=[1,\dots,1]^T$ and $\be_s=[0,\dots,0,1]^T$. 
\end{thm}

\begin{proof}
Define $\tilde M=M-\bone{\be_s}^T$. Since $M\bone=\bone$ and ${\be_s}^T M={\be_s}^T$, ${\tilde M}\bone=\bzero$ and $\be_s^T {\tilde M}=\bzero^T$.  An inductive argument shows that $M^k={{\tilde M}}^k+\bone{\be_s}^T$ for $k\geq 1$.  Furthermore, the spectrum of ${\tilde M}$ is obtained from $\spec(M)$ by replacing eigenvalue 1 with 0 (subtracting $\bone{\be_s}^T$ has the effect of deflating $M$ on eigenvalue 1, as is done in the proof of \cite[Theorem 8.2.7]{HJ13}).  Recall that 
$\ept(G,B)=\sum_{r=1}^\infty r\lp M^r- M^{r-1}\rp_{1s}$ \cite{GH18-PZF}, so we consider $\sum_{r=1}^{\ell} r\lp M^r- M^{r-1}\rp$ as $\ell\to\infty$.
\bea
\sum_{r=1}^{\ell} r(M^r - M^{r-1})
&=&{{\tilde M}}+\bone{\be_s}^T - I + \sum_{r=2}^{\ell} r\lp{{\tilde M}}^r +\bone{\be_s}^T- \lp {\tilde M}^{r-1}+\bone{\be_s}^T\rp\rp\\
&=&\bone{\be_s}^T + \sum_{r=1}^{\ell} r {{\tilde M}}^r -\sum_{r=1}^{\ell-1} r {{\tilde M}}^{r}- \sum_{r=1}^{\ell-1}{{\tilde M}}^{r}-I\\
&=&\bone{\be_s}^T + \ell{{\tilde M}}^{\ell} - \lp I+{\tilde M}+\dots+{{\tilde M}}^{\ell-1}\rp\\
&=&\bone{\be_s}^T + \ell{{\tilde M}}^{\ell} - ({\tilde M}-I)^{-1}({\tilde M}^{\ell}-I).
\eea
Since the spectral radius is less than one,  $\ell{{\tilde M}}^{\ell} \to 0$ and $({\tilde M}-I)^{-1}{{\tilde M}}^{\ell} \to 0$ as $\ell \to \infty$.  Thus,
\[\lim_{\ell\to\infty} \sum_{r=1}^{\ell} r\lp M^r - M^{r-1}\rp=({\tilde M}-I)^{-1} + \bone{\be_s}^T\]
and
\[
\ept(G,B)=\lp(M-\bone{\be_s}^T-I)^{-1} \rp_{1s}+1.
\]%\qedhere
\end{proof}

%{\red try to apply this thm to complete graph or star or linear upper bound to get exact value}

%------------------------------------------------------------------
\subsection{Small graphs}\label{ss:small}

We use Markov matrices and Theorem \ref{t:ept-M} to determine the expected propagation times for all connected graphs of order at most seven.  Expected propagation times for connected graphs of order at most three were known previously and are summarized in Table \ref{tab:ept_ord3-}; when not immediate, a source is given.   

\begin{table}[h!]
\begin{center}
\begin{tabular}[h]{|c|c|c|}
\hline
$G$ &  Source & $\ept(G)$  \\ \hline
$K_1$ &  & 0 \\ \hline %----
$K_2$ &   & 1  \\ \hline %----
$P_3$ & \cite{GH18-PZF} & 2  \\ \hline %----
$K_3$ & \cite{GH18-PZF} & 2  \\ \hline %----
\end{tabular}\vspace{3pt}
\caption{\label{tab:ept_ord3-}Values of $\ept(G)$ for connected graphs of order at most three.}
\end{center}
\end{table}\vspace{-10pt}

Table \ref{tab:ept_ord4} presents the  expected propagation time for each connected graph of order four, including both the exact (rational) value and its decimal approximation,   together with the ordered state list and Markov matrix for an initial vertex that realizes the expected propagation time of the graph.    Data for connected graphs of orders 5, 6, and 7 (omitting the matrices and the exact values) 
can be found in   Appendix 1 \cite{eptsmall} (available online).

\begin{table}[ht!]
\begin{center}
\renewcommand{\arraystretch}{1.07}
\begin{tabular}[h]{|c|c|c|c|}
\hline
$G$ & Labeling & Markov matrix and states &   $\ept(G)$  \\ \hline
$P_4$ & \includegraphics[width=0.12\textwidth]{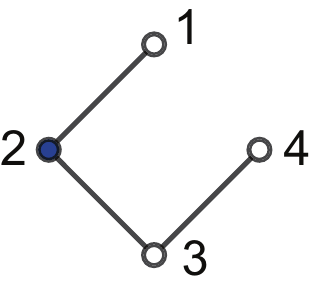}  & %--- start P_4
${\lb\begin{array}{ccccc}
 \frac{1}{4} & \frac{1}{4} & \frac{1}{4} & \frac{1}{4} & 0 \\
 0 & 0 & 0 & 0 & 1 \\
 0 & 0 & 0 & 1 & 0 \\
 0 & 0 & 0 & 0 & 1 \\
 0 & 0 & 0 & 0 & 1 \\
\end{array}
\rb\atop
\{2\}, \{1,2\}, \{2,3\}, \{1,2,3\},  \{1,2,3,4\}}$ & $2\frac 2 3 \approx 2.66667$    \\[1pt] \hline%---- end P_4
$K_{1,3}$ & \includegraphics[width=0.12\textwidth]{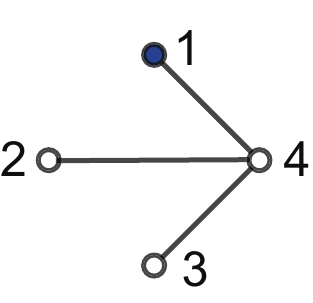} & %--- start K_1,3
${\lb\begin{array}{cccc}
 0 & 1 & 0 & 0 \\
 0 & \frac{1}{9} & \frac{4}{9} & \frac{4}{9} \\
 0 & 0 & 0 & 1 \\
 0 & 0 & 0 & 1 
\end{array}
\rb
\atop{\rm number~of~blue~=~1,~2,~3,~4}}$ & $2\frac{5}{8} \approx 2.625$    \\[2pt] \hline %---end K_1,3
paw  & \includegraphics[width=0.12\textwidth]{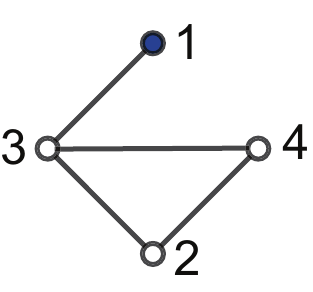} & %--- start paw
${\lb\begin{array}{cccc}
 0 & 1 & 0 & 0 \\
 0 & \frac{1}{9} & \frac{4}{9} & \frac{4}{9} \\
 0 & 0 & 0 & 1 \\
 0 & 0 & 0 & 1 \\
\end{array}
\rb\atop {\rm number~of~blue~=~1,~2,~3,~4}}$ & $2\frac{5}{8}\approx 2.625$    \\[1pt] \hline %---end paw
$C_4$ & \includegraphics[width=0.12\textwidth]{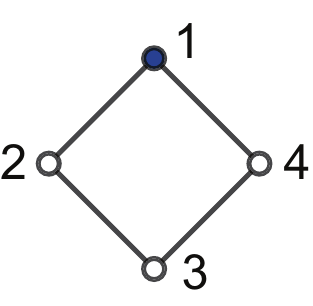} &  %---start C_4
${\lb\begin{array}{cccc}
 \frac{1}{4} & \frac{1}{2} & \frac{1}{4} & 0 \\
 0 & 0 & 0 & 1 \\
 0 & 0 & 0 & 1 \\
 0 & 0 & 0 & 1 \\
\end{array}
\rb\atop {\rm number~of~blue~=~1,~2,~3,~4}}$ &  $2\frac 1 3\approx 2.33333$     \\[3pt] \hline %---end C_4
diamond  & \includegraphics[width=0.12\textwidth]{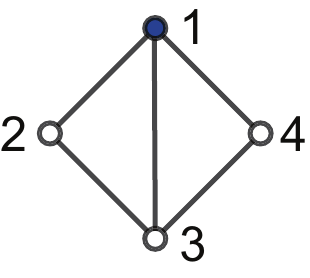} & %--- start diamond
${\lb\begin{array}{ccccc}
\frac{8}{27} & \frac{8}{27} & \frac{4}{27} & \frac{2}{9}
& \frac{1}{27} \\
0 & 0 & 0 & \frac{1}{3} & \frac{2}{3} \\
0 & 0 & \frac{1}{81} & \frac{16}{81} & \frac{64}{81} \\
0 & 0 & 0 & 0 & 1 \\
0 & 0 & 0 & 0 & 1
\end{array}\rb
\atop \{1\},\ \lp\{1,2\} \, {\rm or}\, \{1,4\}\rp,\, \{1,3\},  \ {\rm 3\, blue,~4\, blue}}$ &  $2\frac{631}{1140} \approx 2.55351$    \\[2pt] \hline %--- end diamond
$K_4$ & \includegraphics[width=0.12\textwidth]{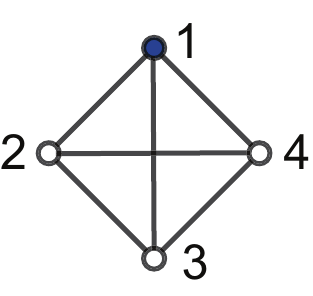} & %--- start K_4
${\lb
\begin{array}{cccc}
 \frac{8}{27} & \frac{4}{9} & \frac{2}{9} & \frac{1}{27} \\
 0 & \frac{1}{81} & \frac{16}{81} & \frac{64}{81} \\
 0 & 0 & 0 & 1 \\
 0 & 0 & 0 & 1 \\
\end{array}
\rb
\atop {\rm number~of~blue~=~1,~2,~3,~4}}$& $2\frac{191}{380} \approx 2.50263$     \\[2pt] \hline %--- end K_4
\end{tabular}\vspace{4pt}
\caption{\label{tab:ept_ord4}Values of $\ept(G)$ for connected graphs of order four.  }\vspace{-12pt}
\end{center}
\end{table}

 Observe that the expected propagation time of the diamond is higher than that of the $4$-cycle, even though the diamond can be obtained by adding an edge to $C_4$.  This demonstrates that adding an edge can raise expected propagation time, thereby answering Question 2.16 in  \cite{GH18-PZF}.
This idea is generalized in Theorem \ref{add-edge-inc-ept} to construct an infinite family of graphs for which adding an edge increases expected propagation time. %\newpage
The {\em tadpole graph} $T_{4,m}$ is constructed from   $C_4$ with vertices $p_1,c_2,c_3,c_4$  labeled cyclically and $P_{m}$  with vertices $p_1,\dots,p_{m}$ labeled in path order as $T_{4,m}=C_4\cup P_{m}$. Form $T'_{4,m}$   by adding the edge $c_2c_4$ to  $T_{4,m}$.  See Figure \ref{f:L45}.
\begin{figure}[h!]\begin{center}
\includegraphics[width=0.3\textwidth]{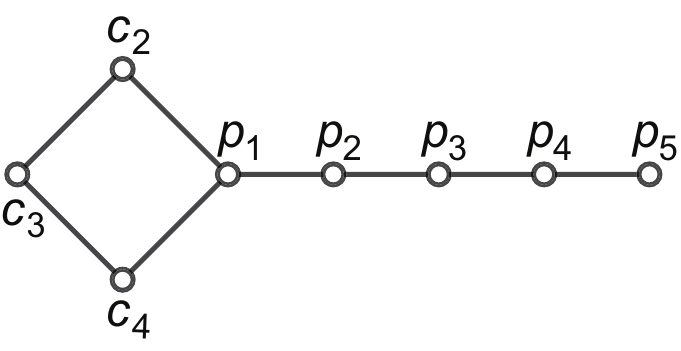}\qquad\includegraphics[width=0.3\textwidth]{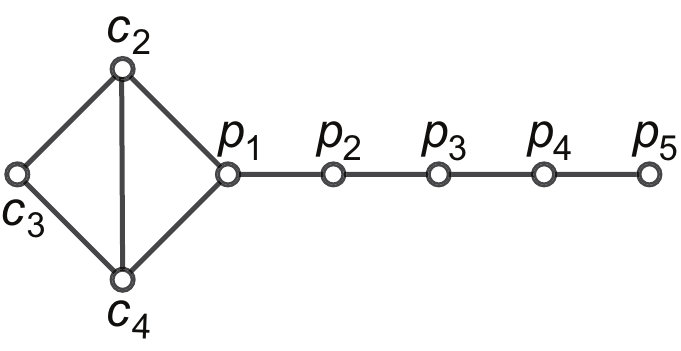}\vspace{-8pt}
\caption{\label{f:L45} The graphs $T_{4,5}$ and $T'_{4,5}$}\vspace{-8pt}
\end{center}\end{figure}

\begin{thm}\label{add-edge-inc-ept}
For infinitely many positive integers $n$, there exist graphs on $n$ vertices such that adding an edge strictly increases the expected propagation time.  Specifically, $\ept(T_{4,m})= \frac {m-1} 2+\frac{451}{216}=\frac {m-1}2 + \frac{1353}{648}$ and $\ept(T'_{4,m})=\frac {m-1}2 + \frac{1429}{648}$ when $m$ is odd, and $\ept(T_{4,m})=  \frac m 2+\frac{3331}{1944}=  \frac m 2+\frac{9993}{5832}$ and $\ept(T'_{4,m})=  \frac m 2+\frac{10357}{5832}$ when $m$ is even.
\end{thm}

\begin{proof}
Suppose that $m \geq 5$. For ease of exposition, we assume that the path is horizontal and to the right of the cycle in $T_{4,m}$ and $T'_{4,m}$, as in Figure \ref{f:L45}. 
First we note that $\ept(T_{4, 2},\left\{p_{1}, p_{2}\right\}) = \frac{17}{8}$, while $\ept(T'_{4, 2},\left\{p_{1}, p_{2}\right\}) = \frac{55}{24}$ (this can be verified by constructing Markov matrices and applying Theorem \ref{t:ept-M}). We define the events $E_0$ and $E_1$ as follows: $E_0$ is the event that after the first force has occurred, in every round in which a non-deterministic force is attempted  there is a successful non-deterministic force.  $E_1$ is the event that after the first force has occurred, in every round but one in which a non-deterministic force is attempted  there is a successful non-deterministic force.  We break the proof into two cases depending on the parity of $m$. 

\noi\underline{Suppose that $m = 2k$ for some positive integer $k$:} \\
\tab First we show that  $\ept(T_{4, m}, \left\{p_{k}\right\}) = \frac{4}{3} + \frac{2}{3}(k-2)+\frac{1}{3}(k-1) + \frac{17}{8} = k+\frac{43}{24}=k+ \frac{10449}{5832}$ and $\ept(T'_{4, m}, \left\{p_{k}\right\}) = \frac{4}{3} + \frac{2}{3}(k-2)+\frac{1}{3}(k-1) + \frac{55}{24} = k+\frac{47}{24}=k+\frac{11421}{5832}$.  In each case the stated value is the expected time for the vertices to the left of $p_k$ to turn blue, consisting of the expected time for the first force, the time after that to deterministically force $p_1$, and  $\ept(T_{4,2},\{p_1,p_2\})$ (respectively, $\ept(T'_{4,2},\{p_1,p_2\})$). The vertices on the right can be ignored because once the first force happens, the time for the vertices to the right of $p_{k}$ to turn blue  is less than or equal to the least possible time for the last vertex on the left of $p_{k}$ to turn blue.  \\   
\tab 
Any vertex other than $p_{k}$ and $p_{k-1}$ has a vertex with distance at least $k+2$ from it in both $T_{4, m}$ and $T'_{4,m}$, which exceeds both $\ept(T_{4, m}, \left\{p_{k}\right\})$ and $\ept(T'_{4, m}, \left\{p_{k}\right\})$, so it suffices to compute $\ept(T_{4, m}, \left\{p_{k-1}\right\})$ and $\ept(T'_{4, m}, \left\{p_{k-1}\right\})$. We split into three cases depending on which vertices are forced in the round where the first force occurs, each of which has probability $\frac{1}{3}$. \\
\tab For the first case, suppose that only $p_{k}$, i.e., the vertex to the right of $p_{k-1}$ gets colored blue in the round with the first force. Then the propagation time for the vertices to the right of $p_{k-1}$ is at most the propagation time for the vertices to the left of $p_{k-1}$, so the expected propagation time in this case is $\frac{4}{3} + (k-2) + \frac{17}{8}= k+\frac{35}{24}$ for $T_{4, m}$ and $\frac{4}{3} + (k-2) + \frac{55}{24}= k+\frac{39}{24}$ for $T'_{4, m}$.\\
\tab For the second case, suppose that both $p_{k}$ and $p_{k-2}$ get colored blue on the first force. Then the propagation time for the vertices to the right of $p_{k-1}$ is at most the propagation time for the vertices to the left of $p_{k-1}$, unless $E_0$ occurs, in which case the propagation time for the vertices to the left of $p_{k-1}$ is one less than the propagation time for the vertices to the right of $p_{k-1}$. Since $\pr(E_0)$ is $\frac{8}{9}$ for $T_{4, m}$ and $\frac{4}{9}+\frac{4}{9} \cdot \frac{2}{3} = \frac{20}{27}$ for $T'_{4, m}$, the expected propagation time in this case is  $\frac{4}{3} + (k-3) + \frac{17}{8} + \frac{8}{9}=k+\frac{291}{216}$ for $T_{4, m}$ and $\frac{4}{3} + (k-3) + \frac{55}{24}+\frac{20}{27}=k+\frac{295}{216}$ for $T'_{4, m}$.\\
\tab For the third case, suppose that only $p_{k-2}$, i.e., the vertex to the left of $p_{k-1}$, gets colored blue in the round with the first force. Then the propagation time for the vertices to the right of $p_{k-1}$ is at most the propagation time for the vertices to the left of $p_{k-1}$, unless $E_0$ or $E_1$ occurs, in which case the propagation time for the vertices to the left of $p_{k-1}$ is two or one less than the propagation time for the vertices to the right of $p_{k-1}$. In both $T_{4, m}$ and $T'_{4, m}$, $\pr(E_0)$ is the same as in the last paragraph. Moreover $\pr(E_1)$ is $\frac{1}{9} \cdot \frac{8}{9}$ for $T_{4, m}$ and $\frac{1}{9}(\frac{20}{27})+\frac{4}{9} \cdot \frac{1}{3}$ for $T'_{4, m}$. Thus the expected propagation time in this case is $\frac{4}{3} + (k-3) + \frac{17}{8} + \frac{8}{9} \cdot 2 + (\frac{1}{9} \cdot \frac{8}{9}) \cdot 1= k + \frac{1513}{648}= k + \frac{13617}{5832}$ for $T_{4, m}$ and $\frac{4}{3} + (k-3) + \frac{55}{24}+\frac{20}{27} \cdot 2 + (\frac{1}{9}(\frac{20}{27})+\frac{4}{9} \cdot \frac{1}{3})\cdot 1= k + \frac{13629}{5832}$ for $T'_{4, m}$.\\
\tab Observe that in each of the three cases, the expected propagation time for $T_{4, m}$ is less than the expected propagation time for $T'_{4, m}$.  We determine $\ept(T_{4, m}, \left\{p_{k-1}\right\})$ and $\ept(T'_{4, m}, \left\{p_{k-1}\right\})$ by averaging over the cases: $\ept(T_{4, m}, \left\{p_{k-1}\right\}) = k + \frac{3331}{1944}= k + \frac{9993}{5832}$ and $\ept(T'_{4, m}, \left\{p_{k-1}\right\}) =k + \frac{10357}{5832}$. 
 Therefore $\ept(T_{4, m}) =  k + \frac{3331}{1944}$ and $\ept(T'_{4, m}) = k + \frac{10357}{5832}$.

\noi\underline{Suppose that $m = 2k+1$ for some positive integer $k$:} \\
\tab This proof is similar to the last proof. Again, we first calculate $\ept(T_{4, m}, \left\{p_{k}\right\})$ and $\ept(T'_{4, m}, \left\{p_{k}\right\})$. Like the proof for $m = 2k$ using $p_{k-1}$, we split the analysis into three cases depending on what happens in the round where the first force occurs. For both cases where the vertex to the right of $p_{k}$ gets colored blue in the round with the first force, the propagation time for the vertices to the right of $p_{k}$ is at most the propagation time for the vertices to the left of $p_{k}$. If both vertices adjacent to $p_{k}$ are colored on the first force, the expected propagation time is $\frac{4}{3} + (k-2) + \frac{17}{8}=k+\frac{35}{24}$ for $T_{4, m}$ and $\frac{4}{3} + (k-2) + \frac{55}{24}=k+\frac{39}{24}$ for $T'_{4, m}$.  If only the vertex to the right of $p_{k}$ gets colored on the first successful force, the expected propagation time is $\frac{4}{3} + (k-1) + \frac{17}{8}=k+\frac{59}{24}$ for $T_{4, m}$ and $\frac{4}{3} + (k-1) + \frac{55}{24}=k+\frac{63}{24}$ for $T'_{4, m}$.\\
\tab For the case where only the vertex to the left of $p_{k}$ gets colored on the first successful force, the propagation time for the vertices to the right of $p_{k}$ is at most the propagation time for the vertices to the left of $p_{k}$, unless $E_0$ occurs. Like the second case of the proof for $m = 2k$, $\pr(E_0)$ is $\frac{8}{9}$ for $T_{4, m}$ and $\frac{4}{9}+\frac{4}{9} \cdot \frac{2}{3} = \frac{20}{27}$ for $T'_{4, m}$. Thus the expected propagation time in this case is  $\frac{4}{3} + (k-2) + \frac{17}{8} + \frac{8}{9}=k+\frac{169}{72}=k+\frac{507}{216}$ for $T_{4, m}$ and $\frac{4}{3} + (k-2) + \frac{55}{24}+\frac{20}{27}=k+\frac{511}{216}$ for $T'_{4, m}$.\\
\tab Again, the expected propagation time for $T_{4, m}$ is less than the expected propagation time for $T'_{4, m}$  in each of the three cases, and we determine $\ept(T_{4, m}, \left\{p_{k-1}\right\})$ and $\ept(T'_{4, m}, \left\{p_{k-1}\right\})$ by averaging over the cases:  $\ept(T_{4, m}, \left\{p_{k}\right\}) = k + \frac{451}{216} < k+2.1$ and $\ept(T'_{4, m}, \left\{p_{k}\right\}) = k + \frac{1429}{648} < k+2.21$. Any vertex besides $p_{k}$ has a vertex with distance at least $k+2$ from it in both $T_{4, m}$ and $T'_{4,m}$, and the probability of failure on the first turn of the coloring process is at least $\frac{1}{4}$ except when $p_m$ is the initial blue vertex, so $\ept(T_{4, m}, \left\{v\right\}) \geq k+2.25$ and $\ept(T'_{4, m}, \left\{v\right\}) \geq k+2.25$ for any $v \neq p_{k}$. Thus  $\ept(T_{4, m}) = k + \frac{451}{216}= k + \frac{1353}{648}$ and $\ept(T'_{4, m}) = k + \frac{1429}{648}$.
\end{proof}

%%%%%%% end Leslie's edits 5/18/19 %%%%%%%%%%%%
%------------------------------------------------------------------
\subsection{The complete graph}

Let $K_n=(V,E)$ be the complete graph on $n$ vertices. Let $B$ be the set of currently blue vertices and let $b=|B| < n$. Consequently, the number of currently white vertices is equal to $n-b$.	For any $v\in B$ and $w\in V \setminus B$, $\pr(v\to w)=\frac{b}{n-1}$ and $\pr(v\not \to w)=1-\frac{b}{n-1}$.
At each given time step, for any given $w\in V \setminus B$, each $v\in B$ will independently attempt to force it. If at least one $v\in B$ is successful, then $w$ is forced. So for any $w\in V \setminus B$ and any integer $k$ such that $0\leq k\leq n-b$, $\pr(\forall v \in B, v\not \to w)=\left(1-\frac{b}{n-1}\right)^b$ and $\pr(B \to w)=1-\left(1-\frac{b}{n-1}\right)^b$. Thus for $b < n-1$,
\beq \label{p:Kn-kforce} \pr(\mbox{exactly $k$ white vertices are forced})={n-b \choose k} \left( 1-\left(\frac{n-1-b}{n-1}\right)^b \right)^k \left(\left(\frac{n-1-b}{n-1}\right)^b\right)^{n-b-k}.\eeq
For $b = n-1$, the process is deterministic (note that \eqref{p:Kn-kforce} remains valid if $0^0 = 1$). The next theorem follows from the previous statement and \eqref{p:Kn-kforce}.

\begin{thm}\label{t:M(K_n)}
	Let $\osl=(S_1,\dots, S_n)$ be the ordered state list where $S_k$ is the state of having $k$ blue vertices in $K_n$. The matrix $M :=M(K_n,\osl)$ has
	\begin{eqnarray*}
		m_{ij}=
		\begin{cases}
			{n-i \choose j-i} \left( 1-\left(\frac{n-1-i}{n-1}\right)^i \right)^{j-i} \left(\left(\frac{n-1-i}{n-1}\right)^i\right)^{n-j} & \mbox{if } i\le \min(n-2,j),\\
			1 & \mbox{if } j = n \text{ and } i = n-1 \text{ or } n,\\
			0 & \mbox{if } i>j \text{ or } i = j = n-1.
		\end{cases}
	\end{eqnarray*}
	Furthermore,  \[
	\spec(M)=\left\{0,1, \bigg(\frac{n-1-i}{n-1}\bigg)^{i(n-i)}:i\in\{1,\dots,n-2\}\right\}.
	\]
\end{thm}

Using Theorems \ref{t:ept-M} and \ref{t:M(K_n)}, we can obtain an exact (rational number) value for the expected propagation time of $K_n$.  However the rational values have rapidly growing numerators and denominators so in the next table we display the decimal equivalents. 

\begin{table}[h!]
	\begin{center}
		\begin{tabular}[h]{||c|r||c|r||c|r||c|r||c|r||}
			\hline
			$n$ & $\ept(K_n)$ & $n$ & $\ept(K_n)$ & $n$ & $\ept(K_n)$ & $n$ & $\ept(K_n)$ & $n$ & $\ept(K_n)$ \\
			\hline%\hline
			1 & 0. & 11  & 3.65014 & 21  & 4.05931 & 31  & 4.24949 & 41 & 4.36583\\
2 & 1. & 12  & 3.71241 & 22  & 4.08432 & 32  & 4.26335 & 42 & 4.3753\\
3 & 2. & 13 & 3.76715 & 23  & 4.1076 & 33  & 4.2766 & 43 & 4.38447\\
4 & 2.50263 & 14  & 3.81611 & 24  & 4.12933 & 34  & 4.2893 & 44 & 4.39334\\
5 & 2.8319 & 15  & 3.8604 & 25  & 4.14966 &  35  & 4.30149 & 45 & 4.40193\\
6 & 3.07164 &16  & 3.90079 & 26  & 4.16874  & 36  & 4.31321 & 46 & 4.41024\\
7 & 3.24769 &17  & 3.93782 &  27  & 4.18671 & 37  & 4.3245 & 47 & 4.4183\\
8 & 3.3829 &18  & 3.97188 &  28  & 4.20367 & 38  & 4.33539 & 48 & 4.42611\\
9 & 3.49035 &19  & 4.00331 &  29  & 4.21973 & 39  & 4.34589 & 49 & 4.43367 \\
10 & 3.57753 & 20  & 4.03238 &  30  & 4.23497 & 40  & 4.35603 & 50 & 4.44101\\ \hline\end{tabular}%\vspace{3pt}
		%\vspace{-3pt}
\caption{\label{tab:ept(K_n)}Values of $\ept(K_n)$ for $n=1,\dots,50$.}
	\end{center}
\end{table}

\noi In Figure \ref{figdata}, we plot this data and the graph of $1.4 \log \log n+2.5$. In Theorem \ref{tkn}, we prove that $\ept(K_n) = \Theta(\log \log n)$.

\begin{figure}[h!]
\centering
\includegraphics[width=0.4\textwidth]{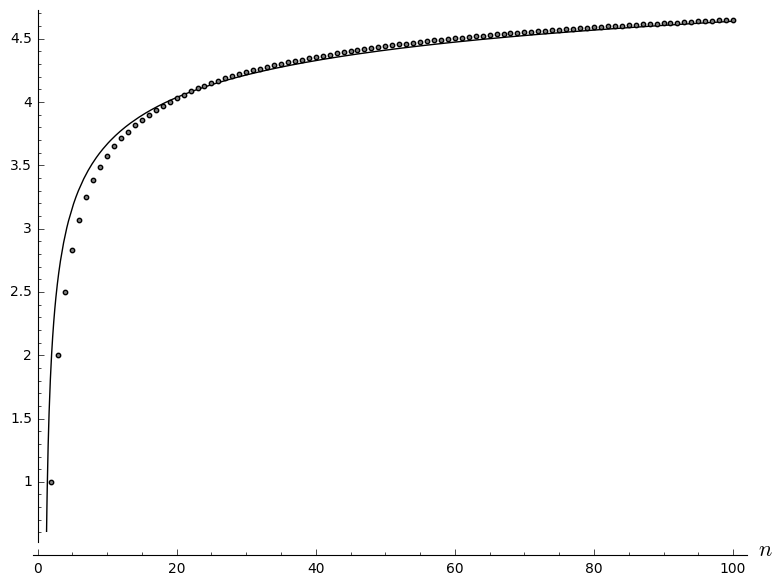}\vspace{-8pt}
\caption{\label{figdata} A plot of $\ept(K_n)$ and $1.4 \log \log n+2.5$.}
\end{figure}
\iffalse
Moreover, we can compute $\ept(K_n)$ by using cofactors: Let $N:=M(K_n)-\bone{\be_n}^T-I_n$ and $a_k:=1-\left(\frac{n-k-1}{n-1}\right)^{k(n-k)}$. Since the principal submatrix $N[\{n-1,n\},\{n-1,n\}]=-I_2$,
\[
(N^{-1})_{1n}=\frac{(-1)^n}{\det{N}}\det(N[\{n\}^c,\{1\}^c])=\left(\prod_{k=1}^{n-2}a_k\right)^{-1}\det(N[\{n-1,n\}^c,\{1,n-1\}^c]).
\]
Notice that $N[\{n-1,n\}^c,\{1,n-1\}^c]$ is an upper Hessenberg matrix.
\fi

%------------------------------------------------------------------
\subsection{The complete bipartite graph}

Using a similar process, we can construct a Markov matrix for the complete bipartite graph $K_{m,n}$. Partition $K_{m,n}$ into its partite vertex sets $R$ and $R'$. We denote each state $(a,b)$, where $a$ and $b$ denote the number of blue vertices in $R$ and $R'$, respectively. In this state, $a$ blue vertices independently attempt to force $n-b \geq 1$ white vertices, each with probability $\frac{b+1}{n}$, and $b$ blue vertices independently attempt to force $m-a \geq 1$ white vertices, each with probability $\frac{a+1}{m}$. 

\begin{prop}
	Given initial state $(a,b)$, the probability of forcing exactly $k$ vertices in $R$ and $\ell$ vertices in $R'$ is
		\[
		{n-b \choose \ell}\left( 1- \left(1-\frac{b+1}{n}\right)^a\right)^{\ell} \left(1-\frac{b+1}{n}\right)^{a(n-b-\ell)}
		{m-a \choose k}\left( 1-\left(1-\frac{a+1}{m}\right)^b\right)^k \left(1-\frac{a+1}{m}\right)^{b(m-a-k)}
		\]
		where we define $0^0=1$.
\end{prop}

\begin{table}[h!]
	\begin{center}
		\begin{tabular}[h]{||c|c|r|r||c|c|r|r||}
			\hline
			$m$ & $n$ &  $\ept(K_{m, n}, \left\{u\right\})$ &  $\ept(K_{m, n}, \left\{v\right\})$ & $m$ & $n$ &  $\ept(K_{m, n}, \left\{u\right\})$ &  $\ept(K_{m, n}, \left\{v\right\})$ \\
			\hline%\hline
1 & 1 & 1.0 & 1.0 & 1 & 8 & 4.81183 & 4.62020\\
1 & 2 & 2.0 & 2.0 & 2 & 8 & 4.18540 & 4.08068\\
2 & 2 & 2.33333 & 2.33333 & 3 & 8 & 4.05503 & 3.98838\\
1 & 3 & 2.76316 & 2.625 & 4 & 8 & 4.01358 & 3.97359\\
2 & 3 & 2.78684 & 2.79028 & 5 & 8 & 4.00381 & 3.98047\\
3 & 3 & 3.02251 & 3.02251 & 6 & 8 & 4.01553 & 4.00292\\
1 & 4 & 3.34171 & 3.2 & 7 & 8 & 4.04534 & 4.04002\\
2 & 4 & 3.21498 & 3.19900 & 8 & 8 & 4.08905 & 4.08905\\
3 & 4 & 3.29626 & 3.29506 & 1 & 9 & 5.06339 & 4.86653\\
4 & 4 & 3.43624 & 3.43624 & 2 & 9 & 4.34793 & 4.23801\\
1 & 5 & 3.80904 & 3.64678 & 3 & 9 & 4.18620 & 4.11382\\
2 & 5 & 3.53899 & 3.48350 & 4 & 9 & 4.12900 & 4.08306\\
3 & 5 & 3.53847 & 3.51642 & 5 & 9 & 4.10694 & 4.07822\\
4 & 5 & 3.59296 & 3.58345 & 6 & 9 & 4.10467 & 4.08725\\
5 & 5 & 3.67540 & 3.67540 & 7 & 9 & 4.11853 & 4.10878\\
1 & 6 & 4.19683 & 4.01910 & 8 & 9 & 4.14588 & 4.14163\\
2 & 6 & 3.79086 & 3.70709 & 9 & 9 & 4.18336 & 4.18336\\
3 & 6 & 3.73857 & 3.69517 & 1 & 10 & 5.28772 & 5.08642\\
4 & 6 & 3.74500 & 3.72317 & 2 & 10 & 4.49207 & 4.37677\\
5 & 6 & 3.78224 & 3.77314 & 3 & 10 & 4.30347 & 4.22654\\
6 & 6 & 3.84289 & 3.84289 & 4 & 10 & 4.23247 & 4.18154\\
1 & 7 & 4.52624 & 4.34043 & 5 & 10 & 4.20132 & 4.16784\\
2 & 7 & 4.00156 & 3.90376 & 6 & 10 & 4.18947 & 4.16770\\
3 & 7 & 3.90751 & 3.84961 & 7 & 10 & 4.19182 & 4.17811\\
4 & 7 & 3.88562 & 3.85339 & 8 & 10 & 4.20632 & 4.19839\\
5 & 7 & 3.89411 & 3.87717 & 9 & 10 & 4.23087 & 4.22732\\
6 & 7 & 3.92618 & 3.91920 & 10 & 10 & 4.26292 & 4.26292\\
7 & 7 & 3.97698 & 3.97698 & & & & \\ \hline\end{tabular}\vspace{2pt}
		%\vspace{-5pt}
\caption{\label{tab:Kmn}Values of $\ept(K_{m, n}, \left\{u\right\})$ and $\ept(K_{m, n}, \left\{v\right\})$ with $u$ in the part of size $m$ and $v$ in the part of size $n$ for $1 \leq m \leq n \leq 10$.}
	\end{center}
\end{table}

If we start with the initial blue vertex $v$ in $R$, we can construct our list of states as 
\[\osl=\{(1,0), (1,1),...,(1,n),(2,0),(2,1),...,(2,n),...,(m,0),(m,1),...,(m,n)\}.\] 

Constructing the matrix and applying Theorem \ref{t:ept-M} for various values of $m$ and $n$  produced the data in Table \ref{tab:Kmn}. Note that $K_{2, 3}$ is an outlier in the sense that $m = 2, n = 3$ is the only pair of values (up to $n = 10$) for which $\ept(K_{m, n})$ is not achieved by choosing a vertex in the larger partite set. Based on this data, we make the following conjecture:

\begin{conj} Let $K_{m,n}$ have partite vertex sets $R$ and $R'$ of orders $m$ and $n$ respectively, and let $u\in R$ and $v\in R'$. If $n>3$ and $m<n$, then $\ept(K_{m,n},\{u\})>\ept(K_{m,n},\{v\})$.
\end{conj}

%------------------------------------------------------------------
\subsection{The sun and comb graphs}

%replace the operators above with these two
%\newcommand{\sun}{\operatorname{-Sun}}
%\newcommand{\comb}{\operatorname{-Comb}}

Let the $n\sun$ be obtained from the $n$-cycle $C_n$ by adding a single leaf to each vertex. Finding the expected propagation time of the $n\sun$ is equivalent to finding the expected propagation time of the embedded cycle $C_n$, and then adding 1 to color all remaining leaves. If we focus primarily on the embedded cycle, then states are determined by the number of blue vertices in the cycle and how many of the outermost leaves have been forced, as the inner leaves have no effect on the cycle propagation.

If the initial blue vertex is on the cycle, we start with the two states involving one blue vertex on the cycle, without or with the adjacent leaf, which we denote $1$ and $1L$, respectively. Next, denote the intermediate states $(c,\ell)$, where $2\leq c\leq n-2$ and $0\leq \ell\leq 2$. Here, $c$ indicates the number of blue vertices on the cycle and $\ell$ indicates the number of outermost leaves forced; notice that intermediate states with the same value of $\ell$ behave similarly to one another. Denote the  last four states $(n-1,\ell)$ and $(n)$, where $n-1$ or $n$ of the cycle vertices are blue, respectively, and $0\le \ell\le 2$.
 %\begin{obs}
All outcomes and probabilities for these states are given in Table \ref{tab:n-sun-props}. %\end{obs}

\begin{table}[h!]
	\begin{center}
	\begin{tabular}{ ||c | c | c || c | c | c ||}
		\hline
		State at time $t$ & State at time $t+1$ & Prob. & State at time $t$ & State at time $t+1$ & Prob.\\ 
		\hline
		1 & 1 & $\frac{8}{27}$ 				& $(c,0)$ & $(c+2,0)$ & $\frac{4}{9}$ \\ \hline
		1 & 1L & $\frac{4}{27}$  				& $(c,1)$ & $(c+1,0)$ & $\frac{1}{9}$ \\ \hline
		1 & (2,0) & $\frac{8}{27}$  			& $(c,1)$ & $(c+1,1)$ & $\frac{2}{9}$ \\ \hline
		1 & (2,1) & $\frac{4}{27}$			& $(c,1)$ & $(c+2,0)$ & $\frac{2}{3}$ \\ \hline
		1 & (3,0) & $\frac{1}{9}$ 			& $(c,2)$ & $(c+2,0)$ & $1$ \\ \hline
		1L & 1L & $\frac{1}{9}$  				& $(n-1,0)$ & $(n-1,0)$ & $\frac{1}{81}$ \\ \hline
		1L & (2,1) & $\frac{4}{9}$  			& $(n-1,0)$ & $(n-1,1)$ & $\frac{4}{81}$ \\ \hline
		1L & (3,0) & $\frac{4}{9}$  			& $(n-1,0)$ & $(n-1,2)$ & $\frac{4}{81}$ \\ \hline
		$(c,0)$ & $(c,0)$ & $\frac{1}{81}$ 		& $(n-1,0)$ & $(n)$ &  $\frac{8}{9}$ \\ \hline
		$(c,0)$ & $(c,1)$ & $\frac{4}{81}$ 		& $(n-1,1)$ & $(n)$ & $1$ \\ \hline
		$(c,0)$ & $(c,2)$ & $\frac{4}{81}$ 		& $(n-1,2)$ & $(n)$ & $1$\\ \hline
		$(c,0)$ & $(c+1,0)$ & $\frac{4}{27}$ 		& $(n)$ & $(n)$ & $1$\\ \hline
		$(c,0)$ & $(c+1,1)$ & $\frac{8}{27}$  		& & & \\ \hline
	\end{tabular}\vspace{2pt}
\caption{\label{tab:n-sun-props} Transition probabilities for the ordered state list  $\osl=\{1,1L,...,(c,0),(c,1),(c,2),...,(n-1,0),(n-1,1),(n-1,2),(n)\}$  of the $n\sun$ as defined above.}\vspace{-5pt}
	\end{center}
\end{table}

%\bigskip
Note that we leave out the fully propagated state. Instead, we add 1 to the propagation time found from the Markov matrix to account for the round needed to force all remaining leaves after reaching state $(n)$. Using Theorem \ref{t:ept-M} and adding 1 for the final round, we can obtain exact values for $\ept(n\sun)$.  Decimal approximations of these values are given in Table \ref{tab:n-sun-ept}. %, which can be found in the online {\red Appendix \cite{eptsmall}}, 
This table also lists the differences in expected propagation time for consecutive $n$, i.e.,  $\Delta \ept(n\sun)= \ept(n\sun)-\ept((n-1)\sun)$.  The clear trend that $\Delta \ept(n\sun)\to 0.6875$ as $n$ becomes large leads to the next conjecture.  %{\red Can't we prove the conjecture?  Discuss.}

\begin{table}[h!]
	\begin{center}
\begin{tabular}{|| c | c | c || c | c | c ||} \hline
    $n$ & $\ept(n\sun)$ & $\Delta \ept(n\sun)$ & $n$
    & $\ept(n\sun)$ & $\Delta \ept(n\sun)$ \\ \hline
    $5$ & $4.77765692007797$ & $0.729718323586744$ & $25$ &
    $18.5143540671558$ & $0.687500000595474$ \\ \hline
    $6$ & $5.44614700021659$ & $0.668490080138619$ & $26$ &
    $19.2018540669176$ & $0.687499999761808$ \\ \hline
    $7$ & $6.14172265492263$ & $0.695575654706038$ & $27$ &
    $19.8893540670129$ & $0.687500000095277$ \\ \hline
    $8$ & $6.82588757375988$ & $0.684164918837255$ & $28$ &
    $20.5768540669748$ & $0.687499999961890$ \\ \hline
    $9$ & $7.51474489939839$ & $0.688857325638504$ & $29$ &
    $21.2643540669900$ & $0.687500000015245$ \\ \hline
    $10$ & $8.20169679288223$ & $0.686951893483841$ & $30$ &
    $21.9518540669839$ & $0.687499999993904$ \\ \hline
    $11$ & $8.88941718576886$ & $0.687720392886632$ & $31$ &
    $22.6393540669863$ & $0.687500000002437$ \\ \hline
    $12$ & $9.57682877299639$ & $0.687411587227531$ & $32$ &
    $23.3268540669854$ & $0.687499999999023$ \\ \hline
    $13$ & $10.2643641949093$ & $0.687535421912946$ & $33$ &
    $24.0143540669858$ & $0.687500000000391$ \\ \hline
    $14$ & $10.9518500135211$ & $0.687485818611719$ & $34$ &
    $24.7018540669856$ & $0.687499999999844$ \\ \hline
    $15$ & $11.6393556888815$ & $0.687505675360446$ & $35$ &
    $25.3893540669857$ & $0.687500000000064$ \\ \hline
    $16$ & $12.3268534181140$ & $0.687497729232458$ & $36$ &
    $26.0768540669856$ & $0.687499999999975$ \\ \hline
    $17$ & $13.0143543265595$ & $0.687500908445543$ & $37$ &
    $26.7643540669856$ & $0.687500000000011$ \\ \hline
    $18$ & $13.7018539631505$ & $0.687499636590999$ & $38$ &
    $27.4518540669856$ & $0.687499999999996$ \\ \hline
    $19$ & $14.3893541085209$ & $0.687500145370441$ & $39$ &
    $28.1393540669856$ & $0.687500000000000$ \\ \hline
    $20$ & $15.0768540503712$ & $0.687499941850303$ & $40$ &
    $28.8268540669856$ & $0.687500000000000$ \\ \hline
    $21$ & $15.7643540736315$ & $0.687500023260217$ & $41$ &
    $29.5143540669856$ & $0.687500000000000$ \\ \hline
    $22$ & $16.4518540643273$ & $0.687499990695837$ & $42$ &
    $30.2018540669856$ & $0.687500000000000$ \\ \hline
    $23$ & $17.1393540680490$ & $0.687500003721681$ & $43$ &
    $30.8893540669856$ & $0.687500000000000$ \\ \hline
    $24$ & $17.8268540665603$ & $0.687499998511324$ & $44$ &
    $31.5768540669856$ & $0.687500000000000$ \\ \hline
    $25$ & $18.5143540671558$ & $0.687500000595474$ & $45$ &
    $32.2643540669856$ & $0.687500000000000$ \\ \hline
	\end{tabular}\vspace{2pt}
\caption{\label{tab:n-sun-ept} Expected propagation times for the $n\sun$, and differences $\Delta \ept(n\sun)= \ept(n\sun)-\ept((n-1)\sun)$ for $n=5,\dots,45$.}\vspace{-5pt}
	\end{center}
\end{table}

\begin{conj}
 $\lim_{n\rightarrow\infty}  \lp\ept(n\sun)-\ept((n-1)\sun)\rp= \frac{11}{16}=0.6875$. 
\end{conj}

We can modify the above process for expected propagation time starting at a leaf rather than on the cycle. If the initial blue vertex is a leaf, the first step is deterministic, yielding state $1L$. Afterwards, the states and probabilities proceed as before. Thus, we simply need to construct the list of states starting at $1L$ instead of $1$, and after finding the expected propagation time from the Markov matrix, add 2 to account for the first and last deterministic steps. In general, this yields a slower expected propagation time, though propagation starting at a leaf still suggests the aforementioned limit of $\frac{11}{16}$.

We can use a similar process to construct the Markov matrix for the $n\comb$, which is obtained from the path $P_n$ by adding a leaf to each vertex.  As the initial blue vertex, choose  $v=\lf\frac{n+1}2\rf$ on the embedded path, which is the center vertex for odd $n$ and  the left center vertex for even $n$. For the comb, we will need to track both the number of vertices forced to the left and to the right of the initial vertex, along with whether or not the outermost leaves are blue. The details, which are similar to the $n\sun$ but messier, are given in   Appendix 2 \cite{ncomb} (available online), along with data.  %{\red [discuss] Not surprisingly, the difference appears to converge to $\frac {11}{16}$.} %Hence, we denote general states $(L,R,\ell,r)$, where $L$ is the number of blue path vertices forced to the left of the initial vertex, $R$ is the number of blue path vertices forced to the right of the initial vertex, and $\ell$ and $r$ denote the number of guaranteed increases to $L$ or $R$ next round, respectively. Like the $n\sun$, we start with initial states, $0$ and $0L$, where no forces have occurred along the embedded path. 

\section{Asymptotic bounds for probabilistic zero forcing}\label{sBounds} 

In this section, we prove asymptotically tight bounds up to a constant factor on several families of graphs, including some that were partially bounded in  \cite{GH18-PZF}. We prove that $\ept(K_n)=\Theta(\log\log n)$. Next we generalize the bound $\ept(K_{1,n}) = \Theta(\log n)$ from \cite{GH18-PZF} by proving that $\ept(K_{c, n}) = \Theta(\log n)$ for constant $c$, where the bound depends on $c$. Generalizing the same bound in a different direction, we show $\Theta(\log n)$ bounds on graphs obtained by adding a universal vertex to a graph of maximum degree at most $c$ (a universal vertex is adjacent to every other vertex). Finally, we prove that $\ept(G) = O(n)$ for all connected graphs $G$ of order $n$.

Geneson and Hogben \cite{GH18-PZF} proved that $\ept(K_n)=\Omega(\log\log n)$. In the next result, we show that bound is tight by proving that $\ept(K_n)=O(\log\log n)$. The method of proof is similar to that used in the proof  in \cite{GH18-PZF} that $\ept(K_{1,n}) = O(\log n)$.

\begin{thm}\label{tkn}
For positive integers $n$, $\ept(K_n)=\Theta (\log\log(n))$.
\end{thm}

\begin{proof}
Let $K_n$ be the complete graph on $n$ vertices for $n \geq 5$. Let $b$ be the number of currently blue vertices and $w=n-b$ be the number of currently white vertices. For each white vertex $v_1,...,v_w$, define the indicator random variable $X_i$ to be 1 if $v_i$ is colored blue in the current round and 0 otherwise, and define $X=\sum_{i=1}^w X_i$. Since the $X_i$'s are i.i.d., we have that $\e[X]=w\e[X_i]=w\left(1-\left(1-\frac{b}{n-1}\right)^b\right)$ and $\var[X] = w\left(1-\left(1-\frac{b}{n-1}\right)^b\right)\left(1-\frac{b}{n-1}\right)^b.$  Since $\left(1-\left(1-\frac{b}{n-1}\right)^b\right)\le \lp1-\lp 1-\frac{b^2}{n-1}\rp\rp=\frac{b^2}{n-1}$  by Bernoulli's inequality, $\var[X]\le \frac w {n-1}b^2\left(1-\frac{b}{n-1}\right)^b \leq b^2$.

For $1\leq b\leq \frac{\sqrt{n}}{\log n}$, we first use binomial expansion on $\e[X]$ to obtain $\e[X] > \frac{wb^2}{n-1}- \sum_{k=1}^{\lf b/2 \rf}w \binom{b}{2k}\left(\frac{b}{n-1}\right)^{2k}$. For each term in the summation, 
\[ w{b\choose 2k}\left(\frac{b}{n-1}\right)^{2k}\leq (n-1)\cdot \frac{b^{2k}}{(2k)!}\cdot \frac{b^{2k}}{(n-1)^{2k}}=\frac{b^2}{(2k)!}\cdot\frac{b^{4k-2}}{(n-1)^{2k-1}}.
\] Since $b=o(\sqrt{n})$, we conclude $b^{4k-2}=o(\sqrt{n}^{4k-2})=o(n^{2k-1})$, and using this, we find \[\frac{b^2}{(2k)!}\cdot \frac{b^{4k-2}}{(n-1)^{2k-1}}=\frac{b^2}{(2k)!}\cdot o(1)=\frac{o(b^2)}{(2k)!}.\] 
Since $\sum_{k=1}^{\infty} \frac{1}{(2k)!}$ converges, this implies 
\[\sum_{k=1}^{\lf b/2 \rf}w{b\choose 2k}\left(\frac{b}{n-1}\right)^{2k}=\sum_{k=1}^{\lf b/2\rf}\frac{o(b^2)}{(2k)!}=o(b^2).\] 
For $b \leq \frac{\sqrt{n}}{\log n}$, we have $w \geq n- \frac{\sqrt{n}}{\log n}$, so $\frac{w}{n-1} = 1-o(1)$. We conclude that \[\e[X]> \frac{wb^2}{n-1}- \sum_{k=1}^{\lf b/2 \rf}w{b\choose 2k}\left(\frac{b}{n-1}\right)^{2k} =b^2-o(b^2).\] 
%{\blu By a Chebyshev's inequality argument similar to the one in the proof of Lemma 2.5 in \cite{GH18-PZF},  $\pr(X\geq \frac{1}{2}b^2)=\Omega(1)$  for $b > 4$.}   {\red comment next 3 lines out:}\\ 
 Since $\e(X) = b^2-o(b^2)$ and $b^2=o(n)$, $\e(X) > \frac 5 6 b^2$ for $n$ sufficiently large. Thus by Chebyshev's inequality, \[\pr(X < \frac 1 2 b^2) \le \pr(|X-E(X)| > \frac 1 3 b^2) \le \frac{\var(X)}{\lp \frac 1 3 b^2\rp^2} \le \frac 9{b^2} \le 9/16\] for $b \ge 4$. %\\
Therefore there exists $c$ such that the expected number of rounds to transition from $b$ blue vertices to at least $\frac{1}{2}b^2$ blue vertices is at most $c$. To establish an upper bound on the expected number of rounds until there are at least $\frac{\sqrt{n}}{\log n}$ blue vertices, consider $f(x)=2^{2^x+1}$, which satisfies $f(k+1)=\frac{1}{2}f(k)^2$. If $2^{2^r+1}=\frac{\sqrt{n}}{\log n}$, then $r=\log_2\left(\log_2\left(\frac{\sqrt{n}}{\log n}\right)-1\right)$. Since the expected time to transition from $1$ to $4$ blue vertices is bounded by a constant, the total expected time to transition from $1$ to $\frac{\sqrt{n}}{\log n}$ blue vertices is at most $cr + O(1)=O(\log\log n)$.

For $\frac{\sqrt{n}}{\log n}\leq b\leq \sqrt{n} \log n$,  Claim (C2) established in the proof of Lemma 2.5 in \cite{GH18-PZF} implies $\pr\left(X\geq \frac{b}{4}\right)=\Omega(1)$. Thus there exists  a constant  $D$ %%% $d$ is used for degree
such that the expected number of rounds to transition from $b$ blue vertices to at least $b+\frac{b}{4}=\frac{5}{4}b$ blue vertices is at most $D$. The expected total rounds to transition from $\frac{\sqrt{n}}{\log n}$ to $\sqrt{n}\log n$ blue vertices is at most  $Dr$, where $r$ is found by solving $\left(\frac{5}{4}\right)^r=\frac{\sqrt{n}\log n}{\sqrt{n}/\log n}$, which gives us $r=2\log_{5/4}\log n$ and $D r=O(\log\log n)$.

 For $n\ge 5$, $\left(\frac{1}{n}\right)^{\log n}\le\frac{1}{n^{\log 5}}<\frac{1}{n^{1.5}}$. So for  $\sqrt{n} \log n\leq b \leq n-2$,  \[\left(1-\frac{b}{n-1}\right)^b  \leq \left(1-\frac{\sqrt{n}\log n}{n}\right)^{\sqrt{n}\log n} <\left(e^{-\log n}\right)^{\log n} < \frac{1}{n^{1.5}}.\]
Note that $X$ ranges from $0$ to $w$, so $w-X$ is nonnegative. This allows us to apply Markov's inequality and linearity of expectation to show \[\pr(X<w-\sqrt{w})  = \pr(w-X>\sqrt{w}) \leq \frac{\e[w-X]}{\sqrt{w}} = \sqrt{w} \left(1-\frac{b}{n-1}\right)^b < \sqrt{w} \cdot \frac{1}{n^{1.5}} <\frac{1}{n}.\]
For the complementary event, we conclude $\pr(X\geq w-\sqrt{w})\geq \frac{n-1}{n}$. Then the expected time to transition from $w$ white vertices to at most $\sqrt{w}$ white vertices is at most $\frac{n}{n-1}$. Hence, the expected number of rounds to transition from $w = n - \sqrt{n}\log n$ to $2$ white vertices is at most $\frac{n}{n-1}\cdot r$, where $r$ is given by $w^{(1/2)^r}=2$. Solving this equation, we find $r=\log_2\log_2w$, implying that $\frac{n}{n-1} \cdot r=\frac{n}{n-1}\cdot \log_2\log_2 w=O(\log\log n)$. Note that for $w \leq 2$, the expected time that remains is bounded by a constant. Thus $\ept(K_n)=\Theta (\log\log(n))$. %we bound the expected time with $2$ white vertices remaining. The probability of all $n-2$ blues failing to force the remaining $2$ is given by $\left(1-\frac{n-2}{n}\right)^{2(n-2)}$. For $n\geq 3$, this is less than $\frac{1}{2}$, implying that the expected time to leave this state is at most $2$. If the propagation is not already complete afterwards, the final step will be deterministic, implying it will take at most $3$ expected rounds to finish propagating when at most $2$ white vertices remain.
\end{proof}

It is known that if a graph $G$ of order $n$ has a universal vertex, then $\ept(G) = O(\log n)$ \cite[Corollary 2.6]{GH18-PZF}. In the next result, we use this fact to prove that $\ept(G) = \Theta(\log n)$ for graphs $G$ obtained by adding a universal vertex to a (not necessarily connected) graph of maximum degree at most $c$.

\begin{thm} Let $c$ be a fixed positive integer and let $\F_c$ be the family of graphs  having maximum degree at most $c$. %Consider the family of graphs $G$ obtained by adding a universal vertex $v$ to a graph $H$ such that 
Let $G$ be a graph of order $n$ with a universal vertex $u$ such that $G - u\in \F_c$.  Then  $\ept(G) = \Theta(\log n)$.
\end{thm}
\bpf
The upper bound follows from   \cite[Corollary 2.6]{GH18-PZF}.  For the lower bound, we consider two cases, based on the    the number $\hat{b}$ of blue vertices when $u$ is colored blue at time $t$.  First, suppose that $\hat{b} \geq \sqrt{n}$.  Since the maximum degree is at most $c$, $\sqrt{n} \leq \hat{b} \leq 1 + c + c^2 + \ldots + c^t = \frac{c^{t+1} - 1}{c-1}$. Thus, $\log_c(\sqrt{n}(c-1)+1) - 1 \leq t$, and we have the desired lower bound. 

If instead $\hat{b} < \sqrt{n}$, we consider the expected number of rounds to transition from at most $\sqrt{n}$ blue vertices to at least $\frac{n}{2}$ blue vertices. Let $X$ be the random variable for the number of new blue vertices in the current round, and let $g(b) = \pr(X\leq 4b+c b)$, where $b$ is the current number of blue vertices. We will show that $g(b) = 1 - O \left(\frac{1}{\sqrt{n}} \right)$ for $\sqrt{n} \leq b \leq \frac{n}{2}$. To this end, note that $X$ is at most the sum of the number of vertices forced by $u$, which we will denote by $s$, plus the number of vertices forced by  vertices other than $u$, which we will denote by $r$. Then, $Pr[s \geq 4b] =  O \left(\frac{1}{\sqrt{n}} \right)$  by the proof of Theorem 2.7 in \cite{GH18-PZF}. Because the maximum degree is at most $c$, we also have $r \leq cb$. Thus, $1 - g(b) = O \left(\frac{1}{\sqrt{n}} \right)$. From this point, the same steps as in the proof of Theorem 2.7 in \cite{GH18-PZF} show that with probability $1-o(1)$, the number of rounds to go from at most $\sqrt{n}$ blue vertices to at least $\frac{n}{2}$ blue vertices is $\Omega(\log n)$ (with the constant dependent on $c$), so  $\ept(G) = \Omega(\log n)$.
\epf

The next result builds on ideas in \cite{GH18-PZF}.

\begin{thm}
For any positive integers $m$ and $n$, $\ept(K_{m,n})=O(\log(m+n))$.  For a fixed positive integer $c$, $\ept(K_{c,n})=\Theta(\log(n))$.
\end{thm}

\begin{proof}
 For the upper bounds: It was shown in  \cite[Lemma 2.5]{GH18-PZF} that $\ept(G[N[v]]) = O(\log \deg v)$ for any vertex $v$.  This implies  $\ept(K_{m,n})=O(\log(m)+\log(n))$.  If $m\le n$, then  $\log(m)+\log(n)\le 2\log n$, so  $\ept(K_{m,n})=O(\log(n))$, which also implies $\ept(K_{m,n})=O(\log(m+n))$ (and no assumption $m\le n$ is needed on the latter). 

 Let $c$ be a fixed positive integer. We consider the lower bound on $\ept(K_{c,n})$. Let $R$ and $R'$ denote the partite sets of orders $c$ and $n$ respectively.   We show first that the expected number of rounds to color all vertices in $R$ blue is $O(1)$.  Suppose first that the one initial blue vertex is in $R$.  By Claim $(C1)$ established in the proof of Lemma 2.5 in \cite{GH18-PZF}, the probability of at least one new blue vertex in a round is at least one half, so  the expected time of the first force is at most 2. Once at least one vertex in $R'$ is blue,  the expected number of rounds to color $R$ blue is at most $\ept(K_{1.c})$.  Thus the expected number of rounds to color $R$ blue is a constant. 
 % Alternative explanation: For any subsets A, B of V(G) with A subset B, we have ept(G, A) >= ept(G,B) as a corollary of prop 4.1.iii in our first pzf paper. Since there is one initial blue vertex for ept(K_{c,n}) (in either R or R?), and we are proving a lower bound on the ept, we can thus assume that the initial blue set has all vertices in R colored and one vertex in R? colored by the inequality in the previous sentence.
 
 So suppose that all the vertices in $R$ are blue and let $b$ denote the current number of blue vertices. For each white vertex $v_1, \ldots, v_{n+c-b} \in R'$, let $X_i$ be the indicator random variable that $v_i$ is colored blue in the current round. Let $X=\sum_{i=1}^{n+c-b}X_i$, and
\[
    \pr(\text{$R \to v_i$})=1-\pr(\text{$\forall u \in R, u \not \to v_i$}) 
    =1-(1-\pr(\text{$u \to v_i$}))^c 
    =1-\left(1-\frac{1+b-c}{n}\right)^c.
\]
Using Bernoulli's inequality  for the first inequality below, we have
\bea
    \e[X]&=&\sum_{i=1}^{n+c-b}\e[X_i] \\
  %  &=&(n+c-b)\e[X_1] \\
    & =&(n+c-b)\left(1-\left(1-\frac{1+b-c}{n}\right)^c\right) \\
 %   &=&(n+c-b)\left(1-\left(1+\left(-\frac{1+b-c}{n}\right)\right)^c\right) \\
    & \leq&(n+c-b)\left(1-\left(1+c\left(-\frac{1+b-c}{n}\right)\right)\right) \\
    &=&\frac{(n+c-b)(1+b-c)c}{n} \\
    &\leq& cb.
\eea
Since the $X_i$ are i.i.d. and $X_i^2 = X_i$,
 \[   \var[X]%&=&\sum_{i=1}^{n+c-b}\var[X_i] \\
   % &=&(n+c-b)\left(\e[X_1^2]-\e[X_1]^2\right)\\
    =(n+c-b)\left(1-\left(1-\frac{1+b-c}{n}\right)^c\right)\left(1-\frac{1+b-c}{n}\right)^c \\
    \leq\e[X] \leq cb.\]

 Consider the case in which $\sqrt{n}\leq b\leq \frac{n}{2}$, and define $h(b)$ to be the probability that the number of new blue vertices in the current round is at most $2cb$. Then  Chebyshev's inequality justifies the third inequality below:
\[
1-h(b)\leq\pr(X-cb\geq cb)\leq \pr(|X-\e[X]|\geq cb)\leq\frac{\var[X]}{(cb)^2}\leq\frac{1}{c\sqrt{n}}=O\left(\frac{1}{c\sqrt{n}}\right).
\]
Starting with $\sqrt{n}\leq b\leq \frac{n}{2}$ blue vertices and coloring at most $2cb$ additional blue vertices per round implies that the probability that there are at most $(3c)^rb$ blue vertices after $r$ rounds is at least $\left(h(b)\right)^r=\left(1-O\left(\frac{1}{c\sqrt{n}}\right)\right)^r$. Thus going from at most $\sqrt{n}$ blue vertices to at least $\frac{n}{2}$ blue vertices requires that $(3c)^r\sqrt{n}\geq \frac{n}{2}$, or $r\geq\log_{3c}\left(\frac{\sqrt{n}}{2}\right)$. Hence the probability is at least $\left(1-O\left(\frac{1}{c\sqrt{n}}\right)\right)^{\log_{3c}\left(\sqrt{n}/2\right)}=1-o(1)$ that it takes at least $\log_{3c}\left(\frac{\sqrt{n}}{2}\right)$ rounds for the number of blue vertices to increase from at most $\sqrt{n}$ to at least $\frac{n}{2}$. So $\ept(K_{c,n})=\Omega(\log(n))$.
\end{proof}

It is shown in \cite{GH18-PZF} that $\ept(G) = O(\rad(G) (\log n)^2)$ for connected graphs $G$ of order $n$. The next result implies that $\ept(G) = O(n)$ for connected graphs $G$ of order $n$. 

\begin{thm}\label{t:linupperbd}
Let $G$ be a connected graph of order $n$. Then $\ept(G, S) \leq  \frac{e}{e-1} (n-|S|)$ for any set $S$ of vertices of $G$. 
\end{thm}

\begin{proof}
We prove this by reverse strong induction on $k = |S|$. It is immediate for $k = n$. Now fix some $k < n$ and suppose that the theorem is true for any $i > k$. Let $S$ be an initial set of blue vertices. Since $G$ is connected, there exists some  $b \in S$ with at least one white neighbor. Let  $d = \deg(b)$, so $d-j+1$ of the neighbors are white for some integer $j$ with $1 \leq j \leq d$. 

Suppose that there have been no forces yet in the graph. The probability that $b$ does not force any of its white neighbors in the current round is at most 
 \[\lp1-\frac{j}{d}\rp^{d-j+1} = \lp1-\frac{j}{d}\rp^{(d/j)(j(d-j+1)/d)}\leq \frac 1 e^{j(d-j+1)/d} \leq \frac{1}{e},\]
where  the first inequality follows from the fact that $(1-\frac 1 x)^x\le \frac 1 e$ for $x\ge 1$ and  the last inequality follows from the fact that $\frac{j(d-j+1)}{d}$ is minimized at $j = 1$ and $j = d$ for all real $j \in [1, d]$. 

If there have not been any forces yet, the probability of a force in the current round is at least $\frac{e-1}{e}$, so the expected number of rounds until the first force is at most $c = \frac{e}{e-1}$. After the first force, there are at least $k+1$ blue vertices. Therefore $\ept(G, S) \leq c + c(n-k-1) \leq c (n-k)$ by the induction hypothesis.
\end{proof}

\begin{cor}
If $G$ is a connected graph on $n$ vertices, then $\ept(G) = O(n)$.
\end{cor}

%%%%%%%%%%%%%%%%%%%%%%%%%%%%%%%%%%%%%%%%%%%

\end{document}